\documentclass[graybox]{svmult}
\usepackage{mathptmx}       
\usepackage{helvet}         
\usepackage{courier}        
\usepackage{type1cm}        

\usepackage{makeidx}         
\usepackage{graphicx}        

\usepackage{multicol}        
\usepackage[bottom]{footmisc}

\usepackage{amsmath}
\usepackage{amssymb}
\usepackage{subfigure}

\DeclareMathOperator*{\spann}{span}\DeclareMathOperator{\supp}{supp}
\DeclareMathOperator{\vol}{vol}

\newcommand{\R}{\mathbb{R}}

\newcommand{\nm}[1]{\|#1\|}
\newcommand{\ip}[2]{\langle#1,#2\rangle}

\DeclareMathOperator{\pfp}{PFP}

\makeindex             

\begin{document}

\title*{Probabilistic frames: An overview}
\author{Martin Ehler and Kasso A.~Okoudjou}
\institute{Martin Ehler \at Helmholtz Zentrum M\"unchen, Institute of Biomathematics and Biometry, Ingolst\"adter Landstr.~1, 85764 Neuherberg, Germany, \email{martin.ehler@helmholtz-muenchen.de}
\at Second affiliation: National Institutes of Health, Eunice Kennedy Shriver National Institute of Child Health and Human Development, Section on Medical Biophysics, 9 Memorial Drive, Bethesda, MD 20892, \email{ehlermar@mail.nih.gov}
\and Kasso A.~Okoudjou \at University of Maryland, Department of Mathematics, Norbert Wiener Center, College Park, MD 20742, \email{kasso@math.umd.edu}}

\maketitle

\abstract{Finite frames can be viewed as mass points distributed in $N$-dimensional Euclidean space. As such they form a subclass of a larger and rich class of probability measures that we call probabilistic frames. We derive the basic properties of probabilistic frames, and we characterize one of their  subclasses in terms of minimizers of some appropriate potential function. In addition, we survey a range of areas where probabilistic frames, albeit, under different names, appear. These areas include directional statistics, the geometry of convex bodies, and the theory of t-designs.}


\section{Introduction}\label{intro} 
Finite frames in $\R^N$ are  spanning sets that allow the analysis and synthesis of vectors in a way similar to basis decompositions. However, frames are redundant systems and as such the reconstruction formula they provide is not unique. This redundancy plays a key role in many applications of frames which appear now in a range of areas that include, but is not limited to, signal processing, quantum computing, coding theory, and sparse representations, cf.~\cite{Christensen:2003aa, koche1, koche2} for an overview.

By viewing the frame vectors as discrete mass distributions on $\R^N$, one can extend frame concepts to probability measures. This point of view was developed in \cite{me11} under the name of probabilistic frames and was further expanded  in \cite{meko11}. The goal of this chapter is to summarize the main properties of probabilistic frames and to bring forth their relationship to other areas of mathematics. 

The richness of the set of probability measures together with the availability of analytic and algebraic tools, make it straightforward to construct many examples of probabilistic frames. For instance, by convolving probability measures, we have been able to generate new probabilistic frames from existing ones.  In addition, the probabilistic framework considered in this chapter, allows us to introduce a new distance on frames, namely the Wasserstein distance \cite{vill09}, also known as the Earth Mover's distance \cite{Levina:2001fk}. Unlike standard frame distances  in the literature such as the $\ell_2$-distance, the Wasserstein metric enables us to define a meaningful distance between two frames of different cardinalities.

As we shall see later in Section \ref{relfields}, probabilistic frames are also tightly related to various notions that appeared in areas such as the theory of $t$-designs \cite{Del77}, Positive Operator-Valued Measures (POVM) encountered in quantum computing \cite{albini09, ebdavies, dl70}, and isometric measures used in the study of convex bodies \cite{gimi00, John:1948uq, Milman:1987aa}. In particular, in 1948, F.~John \cite{John:1948uq} gave a characterization of what is known today as unit norm tight frames in terms of an ellipsoid of maximal volume, called John's ellipsoid. The latter and other ellipsoids in some extremal positions, are supports of probability measures that turn out to be probabilistic frames. The connections between frames and convex bodies could offer new insight to the construction of frames, on which we plan to elaborate elsewhere.

Finally, it is worth mentioning the connections between probabilistic frames and statistics. For instance, in directional statistics probabilistic tight frames can be used to measure inconsistencies of certain statistical tests. Moreover, in the setting of $M$-estimators as discussed in \cite{Kent:1988kx,Tyler:1987fk,Tyler:1987uq}, finite tight frames can be derived from maximum likelihood estimators that are used for parameter estimation of probabilistic frames.

This chapter is organized as follows. In Section~\ref{section:characterization} we define probabilistic frames, prove some of their main properties, and give a few examples. In Section~\ref{pfpot} we introduce the notion of the probabilistic frame potential and characterize its minima in terms of tight probabilistic frames. In Section~\ref{relfields} we discuss the relationship between probabilistic frames and other areas such as the geometry of convex bodies, quantum computing, the theory of $t$-designs, directional statistics, and compressed sensing. 

\section{Probabilistic Frames}\label{section:characterization}
\subsection{Definition and basic properties}\label{def&prop}
Let $\mathcal{P}:=\mathcal{P}(\mathcal{B},\R^N)$ denote the collection of probability measures on $\R^N$ with respect to the Borel $\sigma$-algebra $\mathcal{B}$. Recall that the support of $\mu\in\mathcal{P}$ is
\begin{equation*}
\supp(\mu)=\big\{x\in  \R^N  : \mu(U_x)>0, \text{ for all open subsets $x\in U_x\subset  \R^N $}  \big\}.
\end{equation*}
We write $\mathcal{P}(K):=\mathcal{P}(\mathcal{B},K)$ for those probability measures in $\mathcal{P}$ whose support is contained in $K\subset \R^N$. The linear span of $\supp(\mu)$ in $\R^N$ is denoted by $E_{\mu}$.
\begin{definition}\label{defpf}
A Borel probability measure $\mu \in \mathcal{P}$  is a \emph{probabilistic frame} if there exists $0<A\leq B < \infty$ such that 
\begin{equation}\label{pfineq}
 A\|x\|^2 \leq \int_{\R^{N}} |\langle x,y\rangle |^2 d\mu (y) \leq B\|x\|^2,\quad\text{for all $x\in\R^N$.}
 \end{equation}
The constants $A$ and $B$ are called \emph{lower and upper probabilistic frame bounds}, respectively.
When $A=B,$ $\mu$ is called a \emph{tight probabilistic frame}.  If only the upper inequality holds, then we call $\mu$ a \emph{Bessel probability measure}. 
\end{definition}


This notion was introduced in \cite{me11} and was further developed in \cite{meko11}. We shall see later in Section \ref{pfoperator} that probabilistic frames provide reconstruction formulas similar to those known from finite frames. Moreover, tight probabilistic frames are present in many areas including convex bodies, mathematical physics, and statistics, cf.~Section \ref{relfields}. We begin by giving a complete  characterization of probabilistic frames, for which we first need some preliminary definitions.

Let 
\begin{equation}\label{p2space}
\mathcal{P}_{2}:= \mathcal{P}_{2}(\R^{N})=\bigg\{ \mu \in \mathcal{P}: M_{2}^{2}(\mu):=\int_{\R^{N}}\nm{x}^{2}d\mu(x) < \infty\bigg\}
\end{equation} 
be the (convex) set of all probability measure with finite second moments. There exists a natural metric on $\mathcal{P}_{2}$ called the $2$-\emph{Wasserstein metric} and given by

\begin{equation}\label{wmetric}
W_{2}^{2}(\mu, \nu):=\min\bigg\{\int_{\R^N \times \R^N}\nm{x-y}^{2}d\gamma(x, y), \gamma \in \Gamma(\mu, \nu)\bigg\},
\end{equation}
where $\Gamma(\mu, \nu)$ is the set of all Borel probability measures $\gamma$ on $\R^N \times \R^N$ whose marginals are $\mu$ and $\nu$, respectively, i.e., $\gamma(A\times \R^N)=\mu(A)$ and $\gamma(\R^N \times B) = \nu(B)$ for all Borel subset $A, B$ in $\R^N$. The Wasserstein distance represents the ``work'' that is needed to transfer the mass from $\mu$ into $\nu$, and each $\gamma\in\Gamma(\mu, \nu)$ is called a transport plan. We refer to \cite[Chapter 7]{ags05}, \cite[Chapter 6]{vill09} for more details on the Wasserstein spaces.

\begin{theorem}\label{thm1}
A Borel probability measure $\mu \in \mathcal{P}$  is a \emph{probabilistic frame} if and only if $\mu \in \mathcal{P}_{2}$ and $E_{\mu}=\R^N$. Moreover, if $\mu$ is a tight probabilistic frame, then the frame bound $A$ is given by $A=\tfrac{1}{N}M_{2}^{2}(\mu)=\tfrac{1}{N}\int_{\R^{N}}\|y\|^{2}d\mu(y).$
\end{theorem}

\begin{proof}
Assume first that $\mu$ is a probabilistic frame, and let $\{\varphi_{i}\}_{i=1}^{N}$ be an orthonormal basis for $\R^N$. By letting $x=\varphi_i$ in \eqref{pfineq}, we have $A\leq \int_{\R^{N}}|\langle \varphi_{i},y\rangle|^{2}d\mu(y) \leq B$. Summing these inequalities over $i$ leads to $A\leq \tfrac{1}{N}\int_{\R^{N}}\|y\|^{2}d\mu(y) \leq B < \infty$, which proves that $\mu \in \mathcal{P}_{2}$. Note that the latter inequalities also prove the second part of the theorem. 

To prove $E_\mu=\R^N$, we assume that $E_{\mu}^{\bot} \neq \{0\}$ and choose $0\neq x \in E_{\mu}^{\bot}$. The left hand side of~\eqref{pfineq} then yields a contradiction.

For the reverse implication, let $M_2(\mu)<\infty$ and $E_\mu=\R^N$. The upper bound in~\eqref{pfineq} is obtained by a simple application of the Cauchy-Schwartz inequality with $B= \int_{\R^{N}}\|y\|^{2}\, d\mu(y)$. To obtain the lower frame bound, let
\begin{equation*}
A:=\inf_{x\in\R^N} \big(\frac{\int_{ \R^{N} } |\langle x,y\rangle |^2 d\mu (y)}{\|x\|^2} \big) = \inf_{x\in S^{N-1}} \big(\int_{ \R^{N} } |\langle x,y\rangle |^2 d\mu (y)\big).
\end{equation*}
Due to the dominated convergence theorem, the mapping $x\mapsto \int_{ \R^N } |\langle x,y\rangle |^2 d\mu (y)$ is continuous and the infimum is in fact a minimum since the unit sphere $S^{N-1}$ is compact. Let $x_0$ be in $S^{N-1}$ such that 
\begin{equation*}
A=\int_{ \R^N } |\langle x_0,y\rangle |^2 d\mu (y).
\end{equation*} 
We need to verify that $A>0$: Since $E_\mu=\R^N$,  there is $y_0\in\supp(\mu)$ such that $|\langle x_0,y_0\rangle|^2>0$. Therefore, there is $\varepsilon >0$ and an open subset $U_{y_0}\subset  \R^N $ satisfying $y_0\in U_{y_0}$ and $|\langle x,y\rangle|^2>\varepsilon$, for all $y\in U_{y_0}$. Since $\mu(U_{y_0})>0$, we obtain $A\geq \varepsilon \mu(U_{y_0})>0$, which concludes the proof of the first part of the proposition. 
\end{proof}

\begin{remark} A tight probabilistic frame $\mu$ with $M_{2}(\mu)=1$ will be referred to as \emph{unit norm tight probabilistic frame}. In this case the frame bound is $A=\tfrac{1}{N}$ which only depends on the dimension of the ambient space. In fact, any tight probabilistic frame $\mu$ whose support is contained in the unit sphere $S^{N-1}$ is a unit norm tight probabilistic frame.
\end{remark}

In the sequel, the Dirac measure supported at $\varphi\in\R^N$ is denoted by $\delta_\varphi$.
\begin{proposition}\label{frame2pf}
Let $\Phi=(\varphi_i)_{i=1}^M$ be a sequence of nonzero vectors in $\R^N$, and let $\{a_i\}_{i=1}^M$ be a sequence of positive numbers.

\noindent a) $\Phi$ is a frame with frame bounds $0<A\leq B < \infty$ if and only if $\mu_{\Phi}:=\tfrac{1}{M}\sum_{i=1}^{M}\delta_{\varphi_{i}}$  is  a probabilistic frame with bounds $A/M,$ and $B/M$.

\noindent b) Moreover, the following statements are equivalent:
\begin{itemize}

\item[(i)] $\;\;\Phi$ is a (tight) frame.


\item[(ii)] $\;\;\mu^\Phi:=\frac{1}{\sum_{i=1}^M \|\varphi_i\|^2}\sum_{i=1}^{M}\|\varphi_i\|^2\delta_{\frac{\varphi_{i}}{\|\varphi_i\|}}$ is a (tight) unit norm probabilistic frame.

\item[(iii)] $\;\;\frac{1}{\sum_{i=1}^M a_i^2}\sum_{i=1}^M a^2_i\delta_{\frac{\varphi_i}{a_i}}$ is a (tight) probabilistic frame.
\end{itemize} 
\end{proposition}

\begin{proof} Clearly, $\mu_\Phi$ is a probability measure, and its support is the set $\{\varphi_k\}_{k=1}^{M}$, such that,
\begin{equation*}
 \int_{\R^{N}}\ip{x}{y}^{2}d\mu_{\Phi}(y)= \tfrac{1}{M}\sum_{k=1}^{M}\ip{x}{\varphi_k}^{2}.
 \end{equation*}
 Part a) can be easily derived from the above equality, and direct calculations imply the remaining equivalences.
\end{proof} 

\begin{remark}
Though the frame bounds of $\mu_{\Phi}$ are smaller than those of $\Phi$, we observe that the ratios of the respective frame bounds remain the same. 
\end{remark}

\begin{example}\label{example1}
Let $dx$ denote the Lebesgue measure on $\R^N$ and assume that $f$ is a positive Lebesgue integrable function such that $\int_{\R^N}f(x) dx =1$. If $\int_{\R^{N}}\nm{x}^{2}f(x)dx < \infty$, then the measure $\mu$ defined by $d\mu(x)=f(x)dx$ is a (Borel) probability measure that is a probabilistic frame. 
Moreover, if $f(x_1,\ldots,x_N) = f(\pm x_1,\ldots,\pm x_N)$, for all combinations of $\pm$, then $\mu$ is a tight probabilistic frame, cf.~Proposition 3.13 in \cite{me11}. The latter is satisfied, for instance, if $f$ is radially symmetric, i.e., there is a function $g$ such that $f(x)=g(\|x\|)$. 
\end{example}

Viewing frames in the probabilistic setting that we have been developing has several advantages. For instance, we can use measure theoretical tools to generate new probabilistic frames from old ones and, in fact, under some mild conditions, the convolution of probability measures leads to probabilistic frames. Recall that the convolution of $\mu,\nu\in\mathcal{P}$ is the probability measure given by $\mu \ast \nu(A)=\int_{\R^{N}}\mu(A-x)d\nu(x)$ for $A \in \mathcal{B}$. Before we state the result on convolution of probabilistic frames, we need a technical lemma that is related to the support of a probability measure that we consider later. The result is an analog of the fact that adding finitely many vectors to a frame does not change the frame nature, but affects only its bounds. In the case of probabilistic frames, the adjunction of a single point (or finitely many points) to its support does not destroy the frame property, but just changes the frame bounds:

\begin{lemma}\label{adjunctionsupp}
Let $\mu$ be a Bessel probability measure with bound $B>0$. Given $\epsilon\in (0, 1)$ set $\mu_\epsilon= (1-\epsilon)\mu + \epsilon \delta_{0}$. Then $\mu_\epsilon$ is a Bessel measure with bound $B_\epsilon=(1-\epsilon)B$. If in addition $\mu$ is a probabilistic frame with bounds $0<A\leq B < \infty$, then $\mu_\epsilon$ is also a probabilistic frame with bounds $(1-\epsilon)A$ and $(1-\epsilon)B$. 

In particular, if $\mu$ is a tight probabilistic frame with bound $A$, then so is  $\mu_\epsilon$ with bound $(1-\epsilon)A$ 
\end{lemma}

\begin{proof}
$\mu_\epsilon$ is clearly a probability measure since it is a convex combination of probability measures. The proof of the lemma follows from the following equations
\begin{align*}
\int_{\R^{N}}|\ip{x}{y}|^{2}d\mu_{\epsilon}(y) & = (1-\epsilon)\int_{\R^{N}}|\ip{x}{y}|^{2}d\mu(y) + \epsilon \int_{\R^{N}}|\ip{x}{y}|^{2}d\delta_{0}(y)\\
&= (1-\epsilon)\int_{\R^{N}}|\ip{x}{y}|^{2}d\mu(y).
\end{align*}
\end{proof}

We are now ready to understand the action of convolution on probabilistic frames. 

\begin{theorem}\label{conv}
Let $\mu\in \mathcal{P}_{2}$ be a probabilistic frame and let $\nu \in \mathcal{P}_{2}$. If $\supp(\mu)$ contains at least $N+1$ distinct vectors, then $\mu\ast\nu$ is a probabilistic frame.
\end{theorem}

\begin{proof}
We shall use Theorem~\ref{thm1}:
\begin{align*}
M_{2}^{2}(\mu \ast \nu)&=\int_{\R^{N}}\|y\|^{2}d\mu \ast \nu(y)\\
&=\iint_{\R^{N}\times \R^{N}}\|x+y\|^{2}d\mu(x)d\nu(y)\\
& \leq M_{2}^{2}(\mu)+M_{2}^{2}(\nu) + 2M_{2}(\mu)M_{2}(\nu)\\
&=(M_{2}(\mu)+M_{2}(\nu))^{2} < \infty.
\end{align*}
Thus, $\mu\ast\nu\in\mathcal{P}_2$, and it only remains to verify that the support of $\mu\ast\nu$ spans $\R^N$, cf.~Theorem \ref{thm1}. Since $\supp(\mu)$ must span $\R^N$, there are $\{\varphi_j\}_{j=1}^{N+1}\subset \supp(\mu)$ that form a frame for $\R^N$. Due to their linear dependency, for each $x\in \R^N$, we can find $\{c_{j}\}_{j=1}^{N+1}\subset \R$ such that $x=\sum_{j=1}^{N+1} c_{j}\varphi_{j}$ with $\sum_{j=1}^{N+1}c_{j}=0$. For $y\in\supp(\nu)$, we then obtain
\begin{equation*}
x=x+0y=\sum_{j=1}^{N+1}c_{j}u_{j} +\sum_{j=1}^{N+1}c_{j}y=\sum_{j=1}^{N+1}c_{j}(u_{j} + y) \in \spann( \supp(\mu)+\supp(\nu)).
\end{equation*}
Thus, $\supp(\mu)\subset \spann(\supp(\mu)+\supp(\nu))$. Since $\supp(\mu)+\supp(\nu)\subset\supp(\mu\ast\nu)$, we can conclude the proof.
\end{proof}

\begin{remark}
By Lemma~\ref{adjunctionsupp} we can assume without loss of generality that $0 \in \supp(\nu)$. In this case, if $\mu$ is a probabilistic frame such that $\supp(\mu)$ does not contain $N+1$ distinct vectors, then $\mu \ast \nu$ is still a probabilistic frame. Indeed, $0\in \supp(\nu)$,  and $E_\mu=\R^N$ together with the fact that $\supp(\mu)+\supp(\nu)\subset \supp(\mu\ast\nu)$ imply  that $\supp(\mu\ast\nu)$ also spans $\R^N$. 

Finally, if $\mu$ is a probabilistic frame such that $\supp(\mu)$ does not contain $N+1$ distinct vectors, then $\supp(\mu)=\{\varphi_j\}_{j=1}^N$ forms a basis for $\R^N$. In this case, $\mu \ast \nu$ is not a probabilistic frame if $\nu = \delta_{-x}$, where $x$ is an affine linear combination of $\{\varphi_j\}_{j=1}^N$. Indeed, $x=\sum_{j=1}^N c_j\varphi_j$ with $\sum_{j=1}^N c_j = 1$ implies $\sum_{j=1}^N c_j (\varphi_j - x)=0$ although not all $c_j$ can be zero. Therefore, $\supp(\mu\ast\nu)=\{\varphi_j-x\}_{j=1}^N$ is linearly dependent and, hence, cannot span $\R^N$.
\end{remark}

\begin{proposition}\label{tpfconv}
Let $\mu$ and $\nu$ be tight probabilistic frames. If $\nu$ has zero mean, i.e., $\int_{\R^{N}}y d\nu(y)=0$, then $\mu \ast \nu$ is also a tight probabilistic frame. 
\end{proposition}

\begin{proof}
Let $A_\mu$ and $A_\nu$ denote the frame bounds of $\mu$ and $\nu$, respectively. 
\begin{align*}
\int_{\R^N}|\ip{x}{y}|^{2}d\mu \ast \nu (y) &= \int_{\R^N}\int_{\R^{N}} |\ip{x}{y+z}|^{2}d\mu(y)\, d\nu(z)\\
& =\int_{\R^N}\int_{\R^{N}}  |\ip{x}{y}|^{2}d\mu(y)d\nu(z) + \int_{\R^N}\int_{\R^{N}}  |\ip{x}{z}|^{2}d\mu(y)d\nu(z)\\
& \qquad  +2\int_{\R^N}\int_{\R^{N}} \langle x,y\rangle \langle x,z\rangle d\mu(y)d\nu(z)\\
& =  A_{\mu}\|x\|^{2}+ A_{\nu} \|x\|^{2} + 2 \langle \int_{\R^{N}} \langle x,y\rangle x d\mu(y) , \int_{\R^{N}} z d\nu(z)\rangle\\
& = (A_\mu + A_\nu)\|x\|^2,
\end{align*}
where the latter equality is due to $\int_{\R^{N}} z d\nu(z)=0$.
\end{proof}

\begin{example}\label{example:tight etc}
Let $\{\varphi_i\}_{i=1}^M\subset \R^N$ be a tight frame, and let $\nu$ be a probability measure with $d\nu(x)=g(\|x\|)dx$ for some function $g$. We have already mentioned in Example \ref{example1} that $\nu$ is a tight probabilistic frame and Proposition \ref{tpfconv} then implies that $\big(\frac{1}{M}\sum_{i=1}^M\delta_{-\varphi_i}\big) \ast \nu=\frac{1}{M}\sum_{i=1}^M f(x-\varphi_i)dx$ is a tight probabilistic frame, see Figure \ref{figure:convolution} for a visualization.
\end{example}

\begin{figure}
\centering
\subfigure[The orthonormal basis is convolved with a gaussian]{
\includegraphics[width=.3\textwidth]{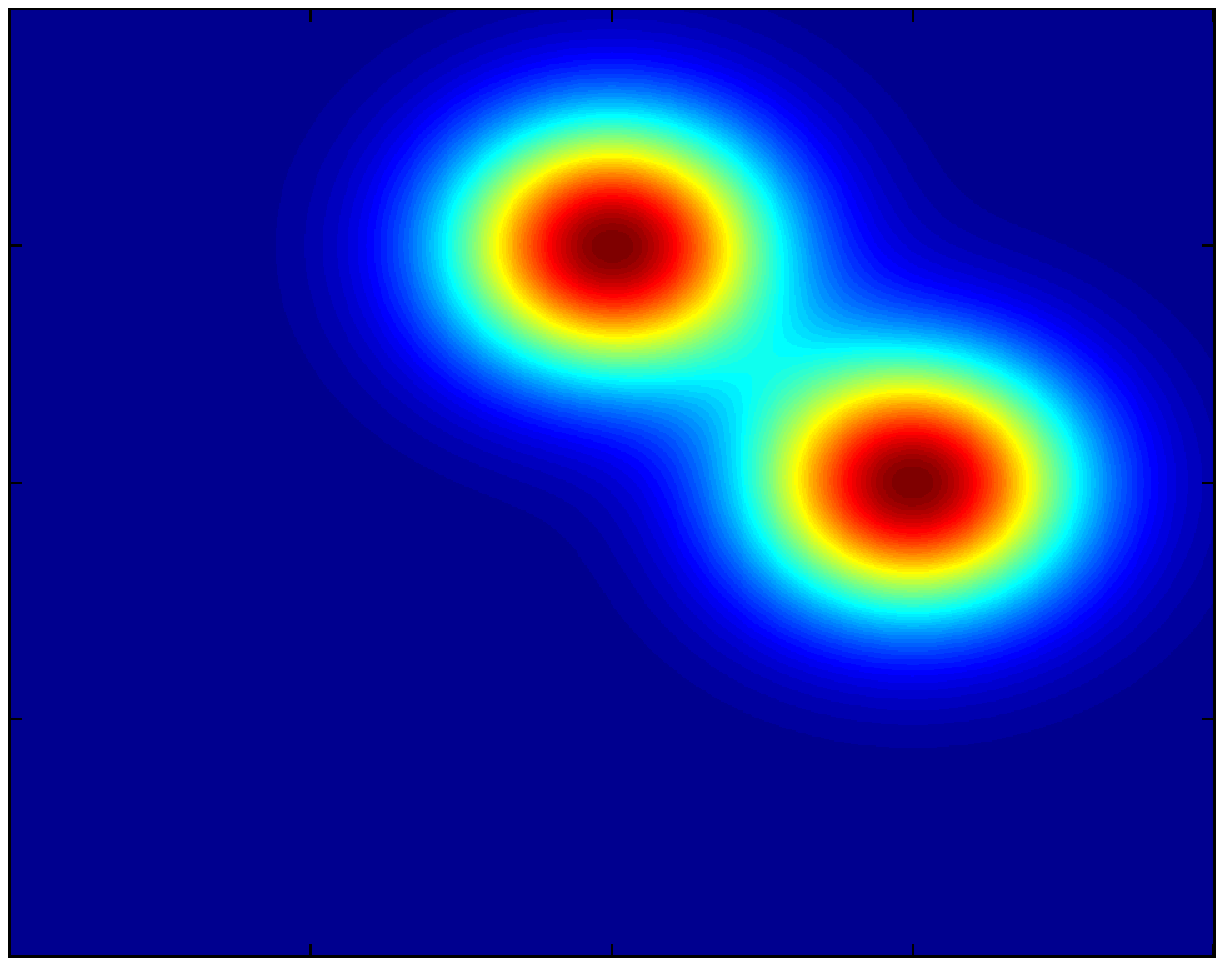}\;\;\;\;\includegraphics[width=.3\textwidth]{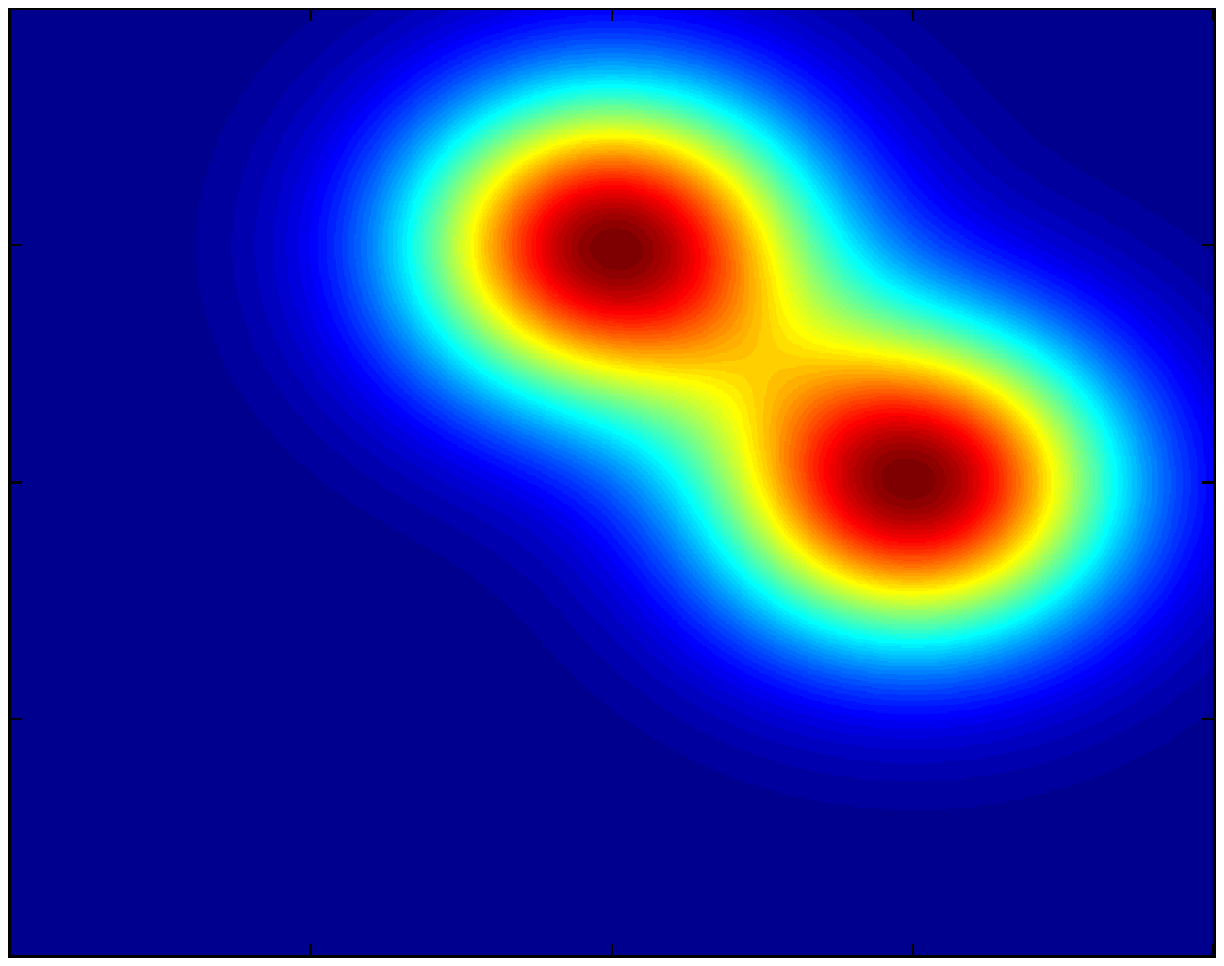}\;\;\;\;\includegraphics[width=.3\textwidth]{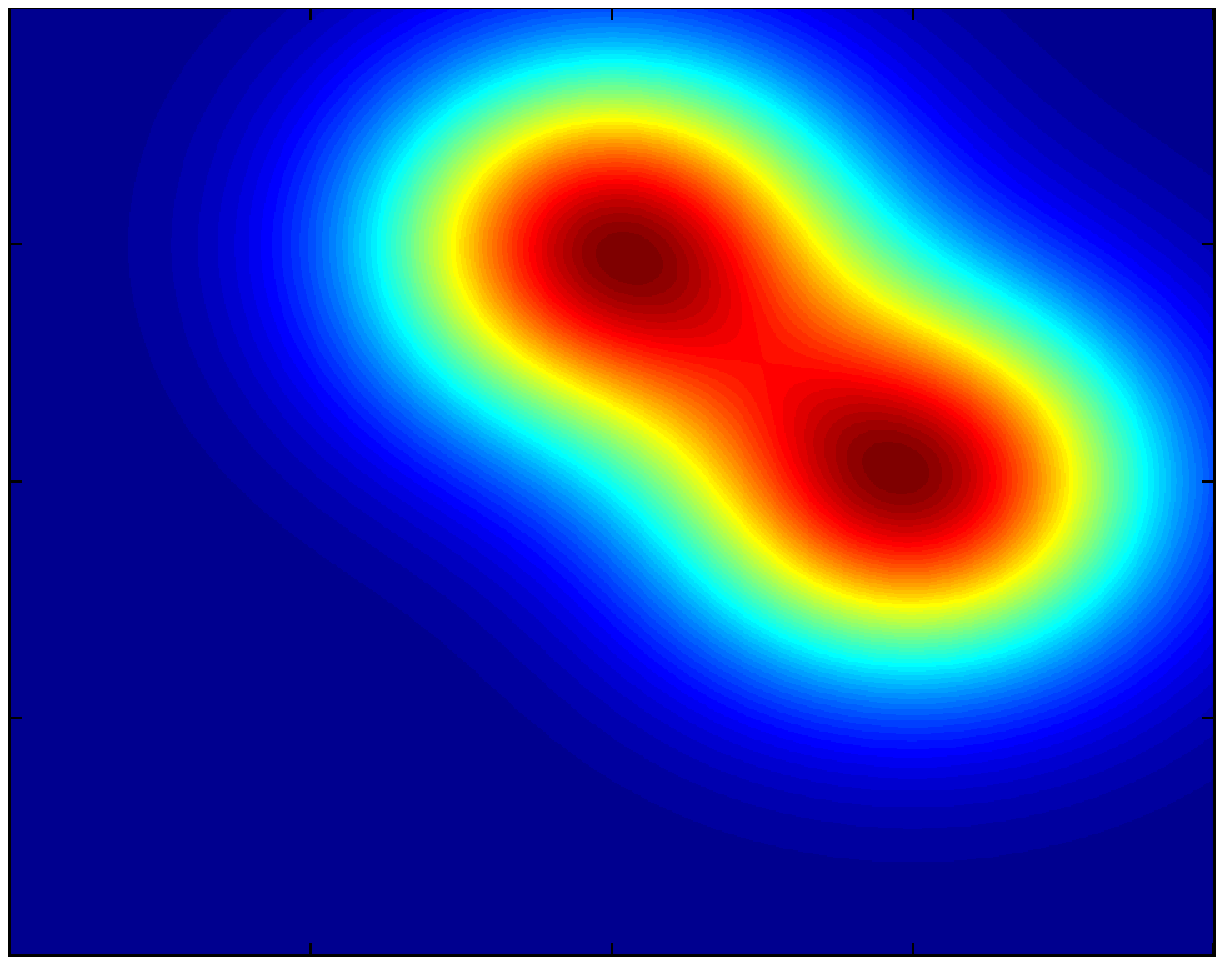}}

\subfigure[The mercedes benz convolved with a gaussian]{
\includegraphics[width=.3\textwidth]{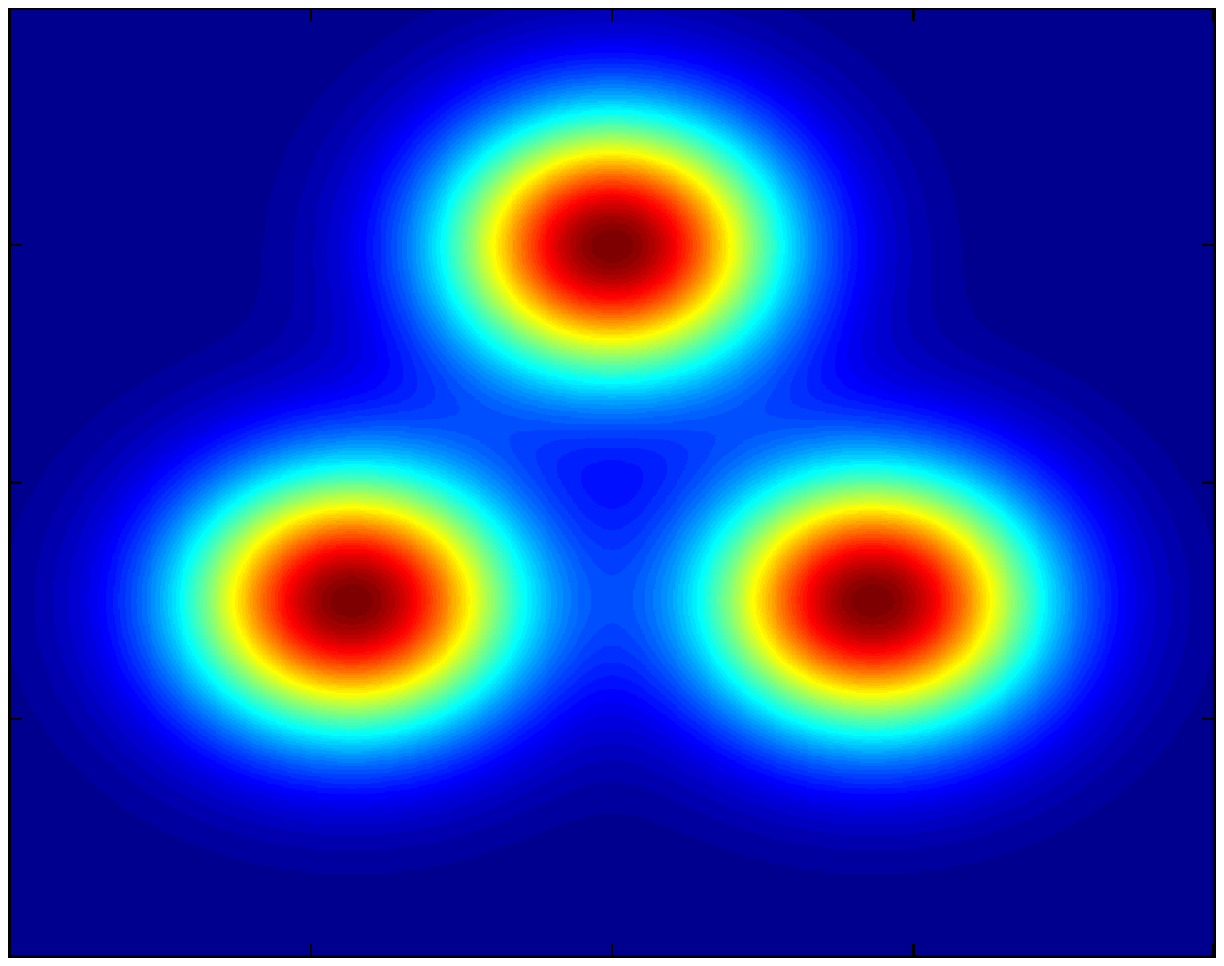}\;\;\;\;\includegraphics[width=.3\textwidth]{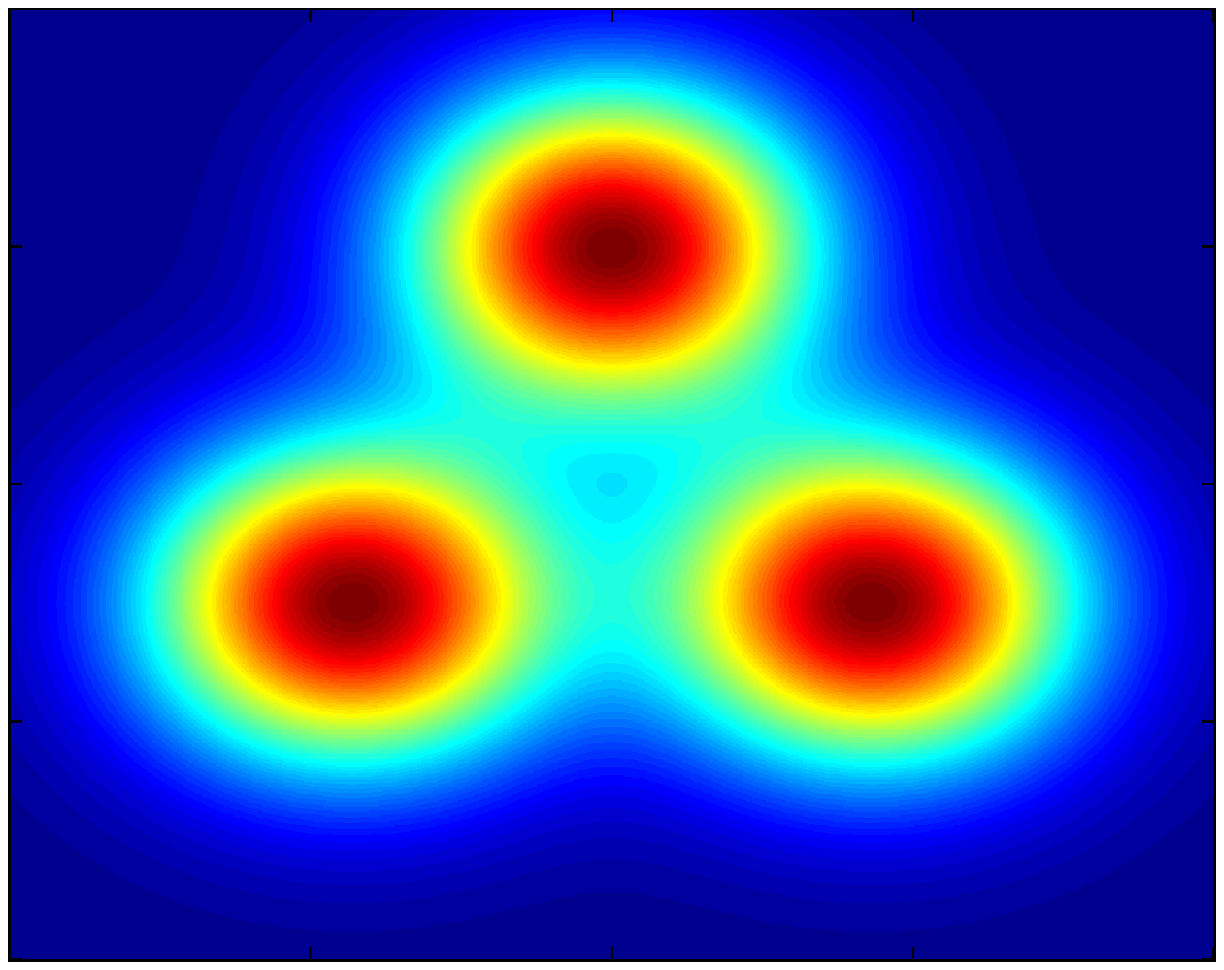}\;\;\;\;\includegraphics[width=.3\textwidth]{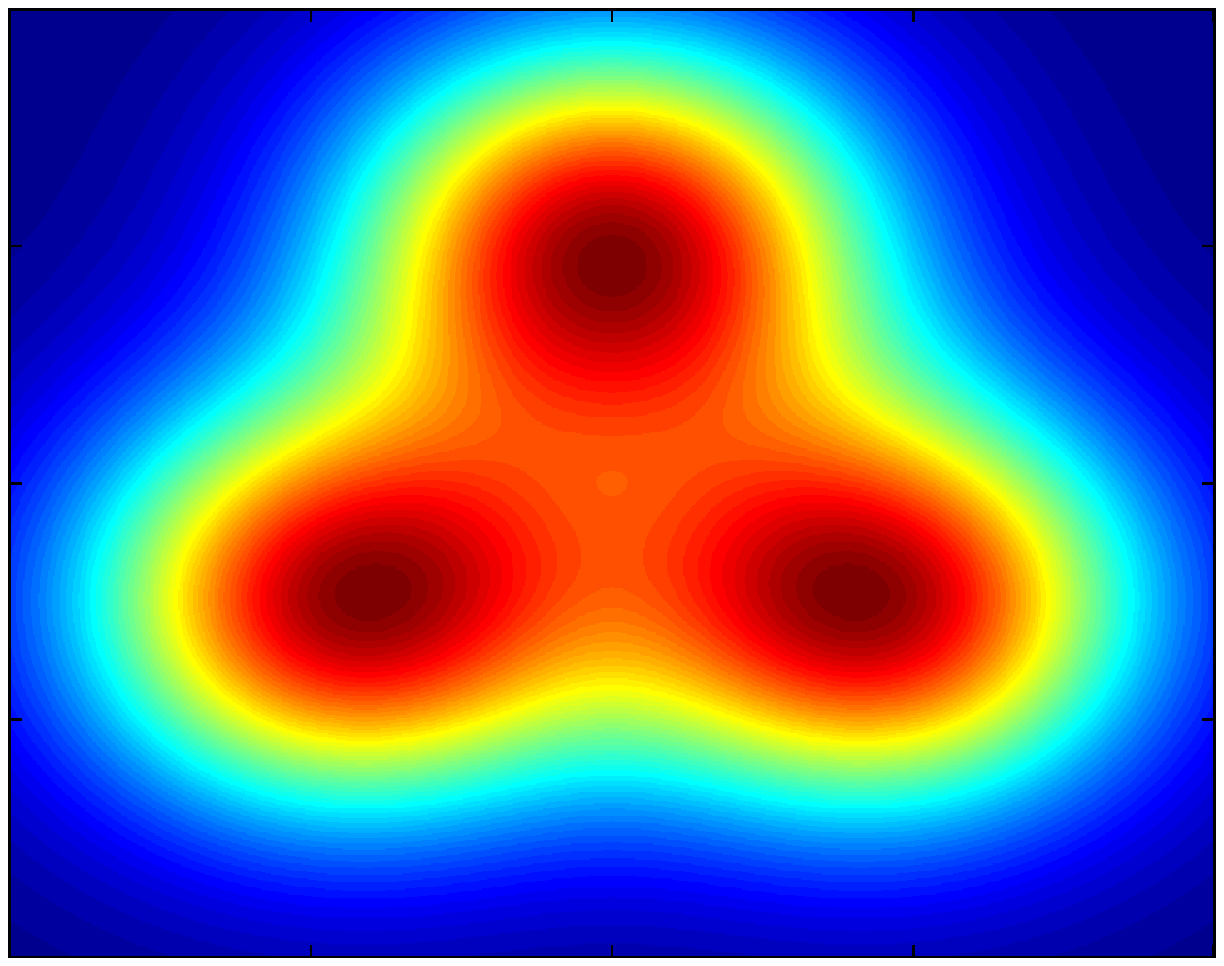}}
\caption{Heatmaps for the associated probabilistic tight frame, where  $\{\varphi_i\}_{i=1}^n\subset \R^2$ is convolved with a gaussian of increased variance (from left to right). The origin is at the center and the axes run from $-2$ to $2$. Each colormap separately scales from zero to the respective density's maximum.}\label{figure:convolution}
\end{figure}

\begin{proposition}\label{minw2}
Let $\mu$ and $\nu$ be two probabilistic frames on $\R^{N_1}$ and $\R^{N_2}$ with lower and upper frame bounds $A_\mu, A_\nu$ and $B_\mu, B_\nu$, respectively, such that at least one of them has zero mean. Then the product measure $\gamma=\mu \otimes \nu$ is a probabilistic frame for $\R^{N_1} \times \R^{N_2}$ with lower and upper frame bounds $\min(A_\mu,A_\nu)$ and $\max(B_\mu,B_\nu)$, respectively.

If, in addition, $\mu$ and $\nu$ are tight and $M^2_{2}(\mu)/N_1=M^2_{2}(\nu)/N_2$, then $\gamma=\mu \otimes \nu$ is a tight probabilistic frame.
 \end{proposition}
\begin{proof}
Let $(z_1, z_2) \in \R^{N_1} \times \R^{N_2}$, then
\begin{align*}
\iint_{\R^{N_1}\times \R^{N_2}} \ip{(z_1, z_2)}{(x, y)}^{2}d\gamma(x, y)& = \iint_{\R^{N_1}\times \R^{N_2}} (\ip{z_1}{x} + \ip{z_2}{y})^{2}d\gamma(x, y)\\
& = \iint_{\R^{N_1}\times \R^{N_2}} \ip{z_1}{x}^{2} d\gamma(x, y) + \iint_{\R^{N_1}\times \R^{N_2}} \ip{z_2}{y}^{2} d\gamma(x, y) \\
&\qquad \qquad + 2\iint_{\R^{N_1}\times \R^{N_2}} \ip{z_1}{x}\ip{z_2}{y}d\gamma(x, y)\\
&= \int_{\R^{N_1}} \ip{z_1}{x}^{2} d\mu(x) + \int_{\R^{N_2}} \ip{z_2}{y}^{2} d\nu(y) \\
& \qquad \qquad + 2\int_{\R^{N_1}}\int_{\R^{N_2}} \ip{z_1}{x}\ip{z_2}{y}d\mu(x)d\nu(y)\\
& =  \int_{\R^{N_1}} \ip{z_1}{x}^{2} d\mu(x) + \int_{\R^{N_2}} \ip{z_2}{y}^{2} d\nu(y)
%
%
\end{align*}
where the last equation follows from the fact one of the two probability measures has zero mean. Consequently, 
\begin{equation*}
A_\mu \|z_1\|^2 + A_\nu \|z_2\|^2 \leq \iint_{\R^{N_1}\times \R^{N_2}} \ip{(z_1, z_2)}{(x, y)}^{2}d\gamma(x, y) \leq B_\mu \|z_1\|^2 + B_\nu \|z_2\|^2, 
\end{equation*} 
and the first part of the proposition follows from $\|(z_1,z_2)\|^2=\|z_1\|^2+\|z_2\|^2$. The above estimate and Theorem \ref{thm1} imply the second part.
\end{proof}

When $N_1=N_2=N$ in Proposition \ref{minw2} and $\mu$ and $\nu$ are tight probabilistic frames for $\R^N$ such that at least one of them has zero mean, then $\gamma=\mu \otimes \nu$ is a tight probabilistic frame for $\R^{N}\times \R^N$. It is obvious that the product measure $\gamma=\mu \otimes \nu$ has marginals $\mu$ and $\nu$, respectively, and hence is an element in $\Gamma(\mu, \nu)$, where this last set was defined in~\eqref{wmetric}. One could ask whether there are any other tight probabilistic frames in $\Gamma(\mu, \nu)$, and if so, how to find them. 

The following question is known in frame theory as the Paulsen problem, cf.~\cite{Bodmann:2010fk,Cahill:2011,Casazzaa:2010fk}: given a frame $\{\varphi_j\}_{j=1}^M\subset\R^N$, how far is the closest tight frame whose elements have equal norm? The distance between two frames $\Phi=\{\varphi_i\}_{i=1}^M$ and $\Psi=\{\psi_i\}_{i=1}^M$ is usually measured by means of the standard $\ell_2$-distance $\sum_{i=1}^M \|\varphi_i-\psi_i\|^2$.

The Paulsen problem can be recast in the probabilistic setting we have been considering, and this reformulation  seems flexible enough to yield new insights into the problem. Given any nonzero vectors $\Phi=\{\varphi_i\}_{i=1}^M$, there are two natural embeddings into the space of probability measures, namely 
\begin{equation*}
\mu_\Phi = \frac{1}{M}\sum_{i=1}^M\delta_{\varphi_i}\quad\text{and}\quad \mu^\Phi:=\frac{1}{\sum_{j=1}^M \|\varphi_j\|^2} \sum_{i=1}^M\|\varphi_i\|^2\delta_{\varphi_i/\|\varphi_i\|}.
\end{equation*}
The $2$-Wasserstein distance between $\mu_\Phi$ and $\mu_\Psi$ satisfies
\begin{equation}\label{eq:estimate Wasserstein and frame distance}
M\| \mu_\Phi- \mu_\Psi\|^2_{W_2} = \inf_{\pi\in \Pi_M}\sum_{i=1}^M \|\varphi_i-\psi_{\pi(i)}\|^2\leq \sum_{i=1}^M \|\varphi_i-\psi_i\|^2,
\end{equation}
where $\Pi_M$ denotes the set of all permutations of $\{1,\ldots,M\}$, cf.~\cite{Levina:2001fk}. The right hand side of~\eqref{eq:estimate Wasserstein and frame distance} represents the standard distance between frames and is sensitive to the the ordering of the frame elements. However, the Wasserstein distance allows to rearrange elements. More importantly, the $\ell_2$-distance requires both frames to have the same cardinalities. On the other hand, the Wasserstein metric enables us to determine how far two frames of different cardinalities are from each other. Therefore, in trying to solve the Paulsen problem, one can seek the closest tight unit norm frame without requiring that it has the same cardinality.

The second embedding $\mu^\Phi$ can be used to illustrate the above point.
 \begin{example}
If, for $\varepsilon\geq 0$,  
\begin{equation*}
\Phi_\varepsilon=\{(1,0)^\top, \sqrt{\frac{1}{2}}(\sin(\varepsilon),\cos(\varepsilon))^\top, \sqrt{\frac{1}{2}}(\sin(-\varepsilon),\cos(-\varepsilon))^\top\},
\end{equation*}
then $\mu^{\Phi_\epsilon}\rightarrow \frac{1}{2}(\delta_{e_1}+\delta_{e_2})$ in the $2$-Wasserstein metric as $\varepsilon\rightarrow 0$, where $\{e_i\}_{i=1}^2$ is the canonical orthonormal basis for $\R^2$. Thus, $\{e_i\}_{i=1}^2$ is close to $\Phi_\varepsilon$ in the probabilistic setting. Since $\{e_i\}_{i=1}^2$ has only $2$ vectors, it is not even under consideration when looking for any tight frame that is close to $\Phi_\varepsilon$ in the standard $\ell_2$-distance. 
 \end{example}   
 
We finish this subsection with a list of open problems whose solution can shed new light on frame theory. The first three questions are related to the Paulsen problem, cf.~\cite{Bodmann:2010fk,Cahill:2011,Casazzaa:2010fk}, that we have already mentioned above:

\begin{problem}\label{questpaulsen}
\begin{itemize}
\item[(a)] \hspace{.5ex}Given a probabilistic frame $\mu\in\mathcal{P}(S^{N-1})$, how far is the closest probabilistic tight unit norm frame $\nu\in\mathcal{P}(S^{N-1})$ with respect to the $2$-Wasserstein metric and how can we find it? Notice that in this case,  $\mathcal{P}_{2}(S^{N-1})=\mathcal{P}(S^{N-1})$ is a compact set, see, e.g.,  \cite[Theorem 6.4]{krp67}.

\item[(b)] \hspace{.5ex}Given a unit norm probabilistic frame $\mu\in\mathcal{P}_2$, how far is the closest probabilistic tight unit norm frame $\nu\in\mathcal{P}_2$ with respect to the $2$-Wasserstein metric and how can we find it? 

\item[(c)] \hspace{.5ex}Replace the $2$-Wasserstein metric in the above two problems with  different Wasserstein $p$-metrics $W^p_p(\mu,\nu)=\inf_{\gamma\in\Gamma(\mu,\nu)}\int_{\R^N\times\R^N} \|x-y\|^p d\gamma(x,y)$, where $2\neq p\in (1, \infty)$. 

\item[(d)] \hspace{.5ex}Let $\mu$ and $\nu$ be two probabilistic tight frames on $\R^N$, such that at least one of them has zero mean. Recall that  $\Gamma(\mu, \nu)$ is the set of all probability measures on $\R^N \times \R^N$ whose marginals are $\mu$ and $\nu$, respectively. Is the minimizer $\gamma_0\in\Gamma(\mu, \nu)$ for $W_{2}^{2}(\mu, \nu)$ a probabilistic tight frame? Alternatively, are there any other probabilistic tight frames in $\Gamma(\mu, \nu)$ besides the product measure?

\end{itemize}
\end{problem}

\subsection{The probabilistic frame and the Gram operators}\label{pfoperator}
To better understand the notion of probabilistic frames, we consider some related operators that encode all the properties of the measure $\mu$. 
Let $\mu \in \mathcal{P}$ be a probabilistic frame. The \emph{probabilistic analysis operator} is given by
 \begin{equation*}
 T_\mu: \R^N \rightarrow L^2(\R^N,\mu), \quad x\mapsto \langle x,\cdot \rangle.
 \end{equation*}
Its adjoint operator is defined by  
 \begin{equation*}
  T^*_\mu: L^2(\R^N,\mu) \rightarrow\R^N , \quad f\mapsto \int_{ \R^N } f(x)xd\mu(x)
 \end{equation*} and is called 
the \emph{probabilistic synthesis operator}, where the above  integral is vector-valued. The \emph{probabilistic Gram operator}, also called the \emph{probabilistic Grammian} of $\mu$, is $G_{\mu}=T_{\mu}T^{*}_{\mu}$. The \emph{probabilistic frame operator} of $\mu$ is $S_\mu=T^*_\mu T_\mu$, and one easily verifies that 
\begin{equation*}
S_\mu:\R^N\rightarrow \R^N,\qquad S_\mu (x) = \int_{\R^N} y\langle x,y\rangle d\mu(y).
\end{equation*}
If $\{\varphi_j\}_{j=1}^N$ is the canonical orthonormal basis for $\R^N$, then the vector valued integral yields
 \begin{equation*}
 \int_{\R^N}  y^{(i)} y d\mu(y)  = \sum_{j=1}^N \int_{\R^N}  y^{(i)} y^{(j)} d\mu(y) \varphi_j,
 \end{equation*}
 where $y=(y^{(1)},\ldots,y^{(N)})^\top\in\R^N$. 
If we denote the second moments of $\mu$ by $m_{i,j}(\mu)$, i.e.,
\begin{equation*}
m_{i,j}(\mu) = \int_{\R^N} x^{(i)} x^{(j)} d\mu(x),\quad \text{for $i,j=1,\ldots,N$,}
\end{equation*}
then we obtain
 \begin{equation*}
S_\mu \varphi_i = \int_{\R^N}  y^{(i)} y d\mu(y) =  \sum_{j=1}^N\int_{\R^N}  y^{(i)} y^{(j)} d\mu(y) \varphi_j = \sum_{j=1}^N m_{i,j}(\mu) \varphi_j.
 \end{equation*}
Thus, the probabilistic frame operator is the matrix of second moments.

The Grammian of $\mu$ is the kernel operator defined on $L^2 (\R^N, \mu)$ by
\begin{equation*}G_{\mu}f(x)=T_{\mu}T^{*}_{\mu}f(x)=\int_{\R^{N}}K(x,y)f(y)d\mu(y)=\int_{\R^{N}}\langle x,y\rangle f(y)d\mu(y).
\end{equation*} 
It is trivially seen that $G$ is a compact operator on $L^{2}(\R^{N}, \mu)$ and in fact it is trace class and Hilbert-Schmidt. Indeed, its kernel is symmetric, continuous,  and in $L^{2}(\R^{N}\times \R^{N}, \mu\otimes \mu ) \subset 
L^{1}(\R^{N}\times \R^{N}, \mu\otimes \mu ).$  Moreover, for any $f \in L^{2}(\R^{N}, \mu)$, $G_{\mu}f$ is a uniformly continuous function on $\R^N$.

Let us collect some properties of $S_\mu$ and $G_\mu$:
\begin{proposition}\label{propg-s} If $\mu\in\mathcal{P}$, then the following points hold:

\noindent a) $S_\mu$ is well-defined (and hence bounded) if and only if 
\begin{equation*}
 M_2(\mu)<\infty.
\end{equation*}

\noindent b) $\mu$ is a probabilistic frame if and only if $S_\mu$ is well-defined and positive definite. 

\noindent c) The nullspace of $G_\mu$ consists of all functions in $L^{2}(\R^{N}, \mu)$ such that $$\int_{\R^N}y f(y) d\mu(y)=0.$$ Moreover,  the eigenvalue $0$ of $G_\mu$ has infinite multiplicity, that is, its eigenspace is infinite dimensional. 

\end{proposition}

For the sake of completeness, we give a detailed proof of Proposition \ref{propg-s}:
\begin{proof} 
Part a): If $S_\mu$ is well-defined, then it is bounded as a linear operator on a finite dimensional Hilbert space. If $\|S_\mu\|$ denote its operator norm and $\{e_i\}_{i=1}^N$ is an orthonormal basis for $\R^N$, then 
\begin{equation*}
\int_{\R^N}\|y\|^2d\mu(y) = \sum_{i=1}^N \int_{\R^N} \langle e_i,y\rangle \langle y,e_i\rangle d\mu(y) = \sum_{i=1}^N \langle S_\mu(e_i),e_i\rangle \leq \sum_{i=1}^N\|S_\mu(e_i)\|\leq N\|S_\mu\|.
\end{equation*}
On the other hand, if $M_2(\mu)<\infty$, then 
\begin{equation*}
\int_{\R^N} |\langle x,y\rangle|^2d\mu(y) \leq \int_{\R^N} \|x\|^2 \|y\|^2d\mu(y) = \|x\|^2 M_2^2(\mu),
\end{equation*}
and, therefore, $T_\mu$ is well-defined and bounded. So is $T_\mu^*$ and hence $S_\mu$ is well-defined and bounded.

Part b): If $\mu$ is a probabilistic frame, then $M_2(\mu)<\infty$, cf.~Theorem \ref{thm1}, and hence $S_\mu$ is well-defined. If $A>0$ is the lower frame bound of $\mu$, then we obtain
\begin{equation*}
\langle x , S_\mu(x)\rangle =  \int_{\R^N} \langle x,y\rangle \langle x,y\rangle d\mu(y) =   \int_{\R^N} |\langle x,y\rangle|^2 d\mu(y) \geq A\|x\|^2,\quad\text{for all $x\in\R^N$,}
\end{equation*}
so that $S_\mu$ is positive definite.

Now, let $S_\mu$ be well-defined and positive definite. According to part a), $M^2_2(\mu)<\infty$ so that the upper frame bound exists. Since $S_\mu$ is positive definite, its eigenvectors $\{v_i\}_{i=1}^N$ are a basis for $\R^N$ and the  eigenvalues $\{\lambda_i\}_{i=1}^N$, respectively, are all positive. Each $x\in\R^N$ can be expanded as $x=\sum_{i=1}^N a_i v_i$ such that $\sum_{i=1}^N a_i^2 = \|x\|^2$. If $\lambda>0$ denotes the smallest eigenvalue, then we obtain
\begin{equation*}
\int_{\R^N} |\langle x,y\rangle |^2 d\mu(y) = \langle x,S_\mu(x)\rangle = \sum_{i,j} a_i\langle v_i,\lambda_j a_j v_j\rangle = \sum_{i=1}^N a_i^2 \lambda_i\geq \lambda \|x\|^2,
\end{equation*}
so that $\lambda$ is the lower frame bound. 

For part c) notice that $f $ is in the nullspace of $G_\mu$ if and only if 
\begin{equation*}
0 = \int_{\R^N}\langle x,y\rangle f(y)d\mu(y) =\langle x ,\int_{\R^N} yf(y)d\mu(y) \rangle  ,\quad\text{for each $x \in \R^N$.}
\end{equation*} 
The above condition is equivalent to $\int_{\R^N} yf(y)d\mu(y)=0$. The fact that the eigenspace corresponding to the eigenvalue $0$ has infinite dimension follows from general principles about compact operators. 
\end{proof}

A key property of probabilistic frames is that they give rise to a reconstruction formula similar to the one used in frame theory. Indeed, if $\mu \in \mathcal{P}_2$ is  a probabilistic frame, set  $\tilde{\mu}=\mu \circ S_\mu$, and we obtain
\begin{equation}\label{reconsppf}
x= \int_{\R^N} \langle x, y\rangle \, S_{\mu}y \, d\tilde{\mu}(y)= \int_{\R^N} y \, \langle S_\mu y, x\rangle \, d\tilde{\mu}(y),\quad\text{for all $x\in\R^N$.}
\end{equation} 
This  follows from $S_\mu^{-1}S_\mu=S_\mu S_\mu^{-1}=Id$. In fact, if $\mu$ is a probabilistic frame for $\R^N$, then $\tilde{\mu}$ is a probabilistic frame for $\R^N$. Note that if $\mu$ is the counting measure corresponding to a finite unit norm tight frame $(\varphi_i)_{i=1}^M$, then $\tilde{\mu}$ is the counting measure associated to the canonical dual frame of $(\varphi_i)_{i=1}^M$, and Equation \eqref{reconsppf} reduces to the known reconstruction formula for finite frames. These observations motivate the following definition:

\begin{definition}\label{candualppf} If $\mu $ is a probabilistic frame, then $\tilde{\mu}=\mu \circ S_\mu $ is  called the \emph{probabilistic canonical dual frame} of $\mu$.
\end{definition}

Many properties of finite frames can be carried over. For instance, we can follow the lines in \cite{Christensen:2003ab} to derive a generalization of the canonical tight frame:
\begin{proposition}
If $\mu$ is a probabilistic frame for $\R^N$, then $\mu\circ S_\mu^{1/2}$ is a tight probabilistic frame for $\R^N$. 
\end{proposition}

\begin{remark}\label{hilbert}
The notion of probabilistic frames that we developed thus far in finite dimensional Euclidean spaces can be defined on  any infinite dimensional separable real Hilbert space $X$ with norm $\|\cdot \|_X$ and inner product $\ip{\cdot}{\cdot}_X$. We call a Borel probability measure $\mu$ on $X$ a \emph{probabilistic frame for $X$} if there exist $0<A\leq B < \infty$ such that 
\begin{equation*}
A\|x\|^2 \leq \int_{X}|\ip{x}{y}|^2 d\mu(y) \leq B \|x\|^2,\quad\text{for all $x \in X$.} 
\end{equation*}
If $A=B$, then we call $\mu$ a probabilistic tight frame and we will present a complete theory of these probabilistic frames in a forthcoming paper. 
\end{remark}

\section{Probabilistic frame potential}\label{pfpot}
The frame potential was defined in \cite{bf03,me11,Rene04,Wal03}, and we introduce the probabilistic analog:
\begin{definition}\label{def2}
For $\mu \in \mathcal{P}_{2}$, the \emph{probabilistic frame potential} is
\begin{equation}\label{eqpfp}
\pfp(\mu)=\iint_{\R^{N}\times \R^{N}}|\langle x,y\rangle |^{2}\, d\mu(x)\, d\mu(y).
\end{equation}
\end{definition}
Note that $\pfp(\mu)$ is well defined for each $\mu \in \mathcal{P}_2$ and $\pfp(\mu)\leq M_{2}^{4}(\mu)$.

In fact, the probabilistic frame potential is just the Hilbert-Schmidt norm of the operator $G_\mu$, that is 
\begin{equation*}
\nm{G_\mu}_{HS}^{2}=\iint_{\R^{N}\times \R^{N}} \ip{x}{y}^{2}d\mu(x) d\mu(y)=\sum_{\ell=0}^{\infty}\lambda_{\ell}^{2},
\end{equation*}
where $\lambda_{k}:=\lambda_{k}(\mu)$ is the $k$-th eigenvalue of $G_\mu$.  
If $\Phi=\{\varphi_{i}\}_{i=1}^{M}$ $M\geq N$ is a finite unit norm tight frame, and 
$\mu=\tfrac{1}{M}\sum_{i=1}^{M}\delta_{\varphi_{i}}$ is the corresponding probabilistic tight frame, then 
\begin{equation*}
\pfp(\mu)= \tfrac{1}{M^{2}}\sum_{i, j=1}^{M}\ip{\varphi_{i}}{\varphi_{j}}^{2}= \tfrac{1}{M^{2}} \tfrac{M^{2}}{N}=\tfrac{1}{N}.
\end{equation*}

According to Theorem 4.2 in \cite{me11}, we have 
\begin{equation*}
\pfp(\mu) \geq \frac{1}{N} M_2^4(\mu),
\end{equation*}
and, except for the measure $\delta_0$, equality holds if and only if $\mu$ is a probabilistic tight frame.

\begin{theorem}\label{minfp} 
If $\mu \in \mathcal{P}_{2}$ such that $M_{2}(\mu)=1$, then 
\begin{equation}\label{estpfp1}
 \pfp(\mu) \geq 1/n,
 \end{equation}  
 where $n$ is the number of nonzero eigenvalues of $S_\mu$. Moreover,  equality holds if and only if $\mu$ is a probabilistic tight frame for $E_{\mu}$. 
 \end{theorem}
Note that we must identify $E_{\mu}$ with the real $\dim(E_{\mu})$-dimensional Euclidean space in Theorem \ref{minfp} to speak about probabilistic frames for $E_{\mu}$. Moreover, Theorem \ref{minfp} yields that if $\mu \in \mathcal{P}_{2}$ such that $M_{2}(\mu)=1$, then $\pfp(\mu) \geq 1/N$, and equality holds if and only if $\mu$ is a probabilistic tight frame for $\R^N$.

\begin{proof} 
Recall that $\sigma(G_\mu)=\sigma(S_\mu) \cup \{0\}$, where $\sigma(T)$ denotes the spectrum of the operator $T$. Moreover, because $G_\mu$ is compact its spectrum consists only of eigenvalues. Moreover, the condition on the support of $\mu$ implies that the eigenvalues $\{\lambda_k\}_{k=1}^N$ of $S_\mu$ are all positive. Since 
\begin{equation*}
\sigma(G_\mu)=\sigma(S_\mu)\cup \{0\}= \{\lambda_{k}\}_{k=1}^{N} \cup \{0\},
\end{equation*}
the proposition reduces to minimizing $\sum_{k=1}^{N}\lambda_{k}^{2}$ under the constraint $\sum_{k=1}^{N}\lambda_{k}=1$, which concludes the proof.
\end{proof}

\section{Relations to other fields}\label{relfields}
{\bf Probabilistic frames, isotropic measures, and the geometry of convex bodies}\\
A finite nonnegative Borel measure $\mu$ on $S^{N-1}$ is called \emph{isotropic} in \cite{gimi00,lyz07} if 
\begin{equation*}
\int_{S^{N-1}}|\langle x,y\rangle|^{2}\, d\mu(y) = \frac{\mu(S^{N-1})}{N}\qquad \forall\, \, x \in S^{N-1}.
\end{equation*}
Thus, every tight probabilistic frame $\mu\in\mathcal{P}(S^{N-1})$ is an isotropic measure. The term isotropic is also used for special subsets in $\R^N$. Recall that  a subset $K \subset \R^N$ is called a convex body if $K$ is compact, convex, and has nonempty interior. According to \cite[Section 1.6]{Milman:1987aa} and \cite{gimi00}, a convex body $K$ with centroid at the origin and unit volume, i.e., $\int_{K}x dx =0$ and $\vol_{N}(K)=\int_K dx=1$, is said to be in \emph{isotropic position} if there exists a constant $L_K$ such that
\begin{equation}\label{isomeas}
\int_{K}|\langle x,y\rangle|^{2}\, dy = L_{K}\qquad \forall\, \, x \in S^{N-1}.
\end{equation}
Thus, $K$ is in isotropic position if and only if the uniform probability measure on $K$, denoted by $\sigma_K$, is a tight probabilistic frame. The constant $L_K$ must then satisfy $L_K=\tfrac{1}{N} \int_{K} \|x\|^2d\sigma_K(x)$. 

In fact, the two concepts, isotropic measures and being in isotropic position, can be combined within probabilistic frames as follows: given any tight probabilistic frame $\mu\in\mathcal{P}$ on $\R^N$, let $K_\mu$ denote the convex hull of $\supp(\mu)$. Then for each $x \in S^{N}$ we have 
\begin{equation*}
\int_{\R^{N}}|\langle x,y\rangle|^{2}\,d\mu(y)=\int_{\supp(\mu)}|\langle x,y\rangle|^{2}\,d\mu(y)=\int_{K_\mu}|\langle x,y\rangle|^{2}\,d\mu(y)
\end{equation*}
Though, $K_\mu$ might not be a convex body, we see that the convex hull of the support of every tight probabilistic frame is in ``isotropic  position'' with respect to $\mu$. 

In the following, let $\mu\in\mathcal{P}(S^{N-1})$ be a probabilistic unit norm tight frame with zero mean. In this case, $K_\mu$ is a convex body and
\begin{equation*}
\vol_N(K_\mu)\geq \frac{(N+1)^{(N+1)/2}}{N!}N^{-N/2},
\end{equation*}
where equality holds if and only if $K_\mu$ is a regular simplex, cf.~\cite{Ball:1992fk,lyz07}. Note that the extremal points of the regular simplex form an equiangular tight frame $\{\varphi_i\}_{i=1}^{N+1}$, i.e., a tight frame whose pairwise inner products $|\langle \varphi_i,\varphi_j\rangle|$ do not depend on $i\neq j$. Moreover, the polar body $P_\mu:=\{x\in\R^N: \langle x,y\rangle\leq 1,\text{ for all } y\in\supp(\mu)\}$ satisfies
\begin{equation*}
\vol_N(P_\mu) \leq \frac{(N+1)^{(N+1)/2}}{N!}N^{N/2},
\end{equation*}
and, again, equality holds if and only if $K_\mu$ is a regular simplex, cf.~\cite{Ball:1992fk,lyz07}. 

Probabilistic tight frames are also related to inscribed ellipsoids of convex bodies. Note that each convex body contains a unique ellipsoid of maximal volume, called John's ellipsoid, cf.~\cite{John:1948uq}. Therefore, there is an affine transformation $Z$ such that the ellipsoid of maximal volume of $Z(K)$ is the unit ball. A characterization of such transformed convex bodies was derived in \cite{John:1948uq}, see also \cite{Ball:1992fk}: 
\begin{theorem}\label{th:John}
The unit ball $B\subset \R^N$ is the ellipsoid of maximal volume in the convex body $K$ if and only if $B\subset K$ and, for some $M\geq N$, there are $\{\varphi_i\}_{i=1}^M\subset S^{N-1}\cap \partial K$ and positive numbers $\{c_i\}_{i=1}^M$ such that
\begin{itemize}
\item[(a)] $\sum_{i=1}^M c_i \varphi_i = 0$ and
\item[(b)] $\sum_{i=1}^M c_i\varphi_i \varphi_i^\top = I_N$.
\end{itemize} 
\end{theorem}
Note that the conditions (a) and (b) in Theorem \ref{th:John} are equivalent to saying that $\frac{1}{N}\sum_{i=1}^M c_i \delta_{\varphi_i}\in\mathcal{P}(S^{N-1})$ is a probabilistic unit norm tight frame with zero mean.

Last but not least, we comment on a deep open problem in convex analysis: Bourgain raised in \cite{Bourgain:1986aa} the following question: \emph{Is there a universal constant $c>0$ such that for any dimension $N$ and any convex body $K$ in $\R^N$ with $\vol_N(K)=1$, there exists a hyperplane $H\subset\R^N$ for which $\vol_{N-1}(K \cap H)>c$}? The positive answer to this question has become known as the hyperplane conjecture. By applying results in \cite{Milman:1987aa}, we can rephrase this conjecture by means of probabilistic tight frames: \emph{There is a universal constant $C$ such that for any convex body $K$, on which the uniform probability measure $\sigma_K$ forms a probabilistic tight frame, the probabilistic tight frame bound is less than $C$}. Due to Theorem \ref{thm1}, the boundedness condition is equivalent to $M_2^2(\sigma_K)\leq C N$. The hyperplane conjecture is still open, but there are large classes of convex bodies, for instance, gaussian random polytopes \cite{B.Klartag:2009aa}, for which an affirmative answer has been established.

\bigskip

\noindent{\bf Probabilistic frames and positive operator valued measures}\\
Let $\Omega$ be a locally compact Hausdorff space, $\mathcal{B}(\Omega)$ be the Borel-sigma algebra on $\Omega$, and $H$ be a real separable Hilbert space with norm $\|\cdot \|$ and inner product $\ip{\cdot}{\cdot}$. We denote by $\mathcal{L}(H)$ the space of bounded linear operators on $H$.
\begin{definition}\label{povm}
A \emph{positive operator valued measure (POVM)} on $\Omega$ with values in $\mathcal{L}(H)$ is a map $F: \mathcal{B}(\Omega) \rightarrow \mathcal{L}(H)$ such that: 

\begin{enumerate}
\item [(i)] $F(A)$ is positive semi-definite for each $A \in \mathcal{B}(\Omega)$;
\item[(ii)]  $F(\Omega)$ is the identity map on $H$;
\item[(ii)] If $\{A_{i}\}_{i\in I}^\infty$ is a countable family of pairwise disjoint Borel sets in $\mathcal{B}(\Omega)$, then $$ F(\cup_{i\in I} A_i)=\sum_{i\in I} F(A_{i}),$$ where the series on the right-hand side converges in the weak topology of $\mathcal{L}(H)$, i.e., for all vectors $x, y \in H$, the series
$
\sum_{i\in I}\ip{F(A_{i})x}{y}
$
converges.  We refer to \cite{albini09, ebdavies, dl70} for more details on POVMs. 
\end{enumerate}
\end{definition}

In fact, every probabilistic tight frame on $\R^N$ gives rise to a POVM on $\R^N$ with values in the set of real $N\times N$ matrices:
\begin{proposition}\label{ptf2povm}
Assume that $\mu \in \mathcal{P}_{2}(\R^N)$ is a probabilistic tight frame. Define the operator $F$ from $\mathcal{B}$ to the set of real $N \times N$ matrices by 
\begin{equation}\label{epovm}
F(A):= \tfrac{N}{M_{2}^{2}(\mu)} \bigg(\int_{A}y_{i}y_{j}\, d\mu(y)\bigg)_{i,j}.
\end{equation}
Then $F$ is a POVM.

\end{proposition}

\begin{proof}
Note that for each Borel measurable set $A$, the matrix $F(A)$ is positive semi-definite, and we also have $F(\R^{N})= I_{N}$. Finally, for a countable family of pairwise disjoint Borel measurable sets $\{A_{i}\}_{i\in I}$, we clearly have for each $x \in \R^N$, $$F(\cup_{i\in I} A_{i}) x = \sum_{k \in I}F(A_{k})x.$$ Thus, any probabilistic tight frame in $\R^N$ gives rise to a POVM.
\end{proof}

We have not been able to prove or disprove whether the converse of this proposition holds:
\begin{problem}\label{questpovm}
Given a POVM $F : \mathcal{B}(\R^N)\rightarrow \mathcal{L}(\R^N)$, is there a tight probabilistic frame $\mu$ such that $F$ and $\mu$ are related through~\eqref{epovm}? 
\end{problem}

\bigskip

\noindent{\bf Probabilistic frames and $t$-designs}\\
Let $\sigma$ denote the uniform probability measure on $S^{N-1}$. A \emph{spherical $t$-design} is a finite subset $\{\varphi_i\}_{i=1}^M\subset S^{N-1}$,
such that,
\begin{equation*}
\frac{1}{n}\sum_{i=1}^n h(\varphi_i) = \int_{S^{N-1}} h(x)d\sigma(x),
\end{equation*}
for all homogeneous polynomials $h$ of total degree less than or equal to $t$ in $N$ variables, cf.~\cite{Del77}. A probability measure $\mu\in\mathcal{P}(S^{N-1})$ is called a \emph{probabilistic spherical $t$-design} in \cite{me11} if
\begin{equation}\label{eq:prob spherical design}
\int_{S^{N-1}}h(x)d\mu(x) = \int_{S^{N-1}} h(x)d\sigma(x),
\end{equation}
for all homogeneous polynomials $h$ with total degree less than or equal to $t$. The following result has been established in \cite{me11}:
\begin{theorem}\label{theorem spherical and FP}
If $\mu \in\mathcal{P}(S^{N-1})$, then the following are equivalent:
\begin{itemize}
\item[(i)] \hspace{.8ex}$\mu$ is a probabilistic spherical $2$-design.
\item[(ii)] \hspace{.8ex}$\mu$ minimizes
\begin{equation}\label{mixed potential}
\frac{\int_{S^{d-1}} \int_{S^{d-1}} |\langle x,y\rangle |^2 d\mu(x) d\mu(y)}{\int_{S^{d-1}} \int_{S^{d-1}}  \|x-y\|^2 d\mu(x)d\mu(y)}
\end{equation}
among all probability measures $\mathcal{P}(S^{N-1})$. 
\item[(iii)] \hspace{.8ex}$\mu$ is a tight probabilistic unit norm frame with zero mean.
\end{itemize}
In particular, if $\mu$ is a tight probabilistic unit norm frame, then $\nu(A):=\frac{1}{2}(\mu(A)+\mu(-A))$, for $A\in\mathcal{B}$, defines a probabilistic spherical $2$-design. 
\end{theorem}
Note that $\frac{1}{N}\sum_{i=1}^M c_i \delta_{\varphi_i}\in\mathcal{P}(S^{N-1})$ derived from conditions (a) and (b) in Theorem \ref{th:John} on the John ellipsoid is a probabilistic spherical $2$-design.

\bigskip

\noindent{\bf Probabilistic frames and directional statistics}\\
Common tests in directional statistics focus on whether or not a sample on the unit sphere $S^{N-1}$ is uniformly distributed. The \emph{Bingham test} rejects the hypothesis of directional uniformity of a sample $\{\varphi_i\}_{i=1}^M\subset S^{N-1}$ if the \emph{scatter matrix}
 \begin{equation*}
 \frac{1}{M}\sum_{i=1}^M \varphi_i \varphi_i^\top
 \end{equation*}
is far from $\frac{1}{N}I_N$, cf.~\cite{Mardia:2008aa}. Note that this scatter matrix is the scaled frame operator of $\{\varphi_i\}_{i=1}^M$ and, hence, one measures the sample's deviation from being a tight frame. Probability measures $\mu$ that satisfy $S_\mu=\frac{1}{N}I_N$ are called \emph{Bingham-alternatives} in \cite{mejg11}, and the probabilistic unit norm tight frames on the sphere $S^{N-1}$ are the Bingham alternatives.



Tight frames also occur in relation to $M$-estimators as discussed in \cite{Kent:1988kx,Tyler:1987fk,Tyler:1987uq}: The family of angular central Gaussian distributions are given by densities $f_\Gamma$ with respect to the uniform surface measure on the sphere $S^{N-1}$, where 
\begin{equation*}
f_\Gamma (x) = \frac{\det(\Gamma)^{-1/2}}{a_N} (x^\top \Gamma^{-1} x)^{-N/2}, \quad\text{for $x\in S^{N-1}$.}
\end{equation*}
Note that $\Gamma$ is only determined up to a scaling factor. According to \cite{Tyler:1987uq}, the maximum likelihood estimate of $\Gamma$ based on a random sample $\{\varphi_i\}_{i=1}^M\subset S^{N-1}$ is the solution $\hat{\Gamma}$ to
\begin{equation*}
\hat{\Gamma} = \frac{M}{N} \sum_{i=1}^M \frac{\varphi_i \varphi_i^\top}{\varphi_i^\top \hat{\Gamma}^{-1}\varphi_i },
\end{equation*}
which can be found, under mild assumptions, through the iterative scheme
\begin{equation*}
\Gamma_{k+1} = \frac{N}{\sum_{i=1}^M \frac{1}{\varphi_i^\top \Gamma_k^{-1}\varphi_i}} \sum_{i=1}^M \frac{\varphi_i \varphi_i^\top}{\varphi_i^\top \Gamma_k^{-1}\varphi_i},
\end{equation*}
where $\Gamma_0=I_N$, and then $\Gamma_k\rightarrow \hat{\Gamma}$ as $k\rightarrow \infty$. It is not hard to see that $\{\psi_i\}_{i=1}^M:=\big\{\frac{\hat{\Gamma}^{-1/2}\varphi_i}{\|\hat{\Gamma}^{-1/2}\varphi_i\|}\big\}_{i=1}^M\subset S^{N-1}$ forms a tight frame. If $\hat{\Gamma}$ is close to the identity matrix, then $\{\psi_i\}_{i=1}^M$ is close to $\{\varphi_i\}_{i=1}^M$ and it is likely that $f_\Gamma$ represents a probability measure that is close to being tight, in fact, close to the uniform surface measure.

\bigskip

\noindent{\bf Probabilistic frames and compressed sensing}\\
For a point cloud $\{\varphi_i\}_{i=1}^M$, the frame operator is a scaled version of the sample covariance matrix up to subtracting the mean and can be related to the population covariance when chosen at random. To properly formulate a result in \cite{me11}, let us recall some notation. For $\mu\in\mathcal{P}_2$, we define $E(Z):=\int_{\R^N} Z(x) d\mu(x)$, where $Z:\R^N\rightarrow \R^{p\times q}$ is a random matrix/vector that is distributed according to $\mu$. The following was proven in \cite{me11}:
\begin{theorem}\label{theorem:final main result}
Let $\{X_{k}\}_{k=1}^M$ be a collection of random vectors, independently distributed according to probabilistic tight frames $\{\mu_{k}\}_{k=1}^M\subset \mathcal{P}_2$, respectively, whose $4$-th moments are finite, i.e., $M^4_4(\mu_k):= \int_{\R^N} \|y\|^4 d\mu_{k}(y)<\infty$. If $F$ denotes the random matrix associated to the analysis operator of $\{X_{k}\}_{k=1}^M$, then we have
\begin{equation}\label{eq:theorem final}
E(\|\frac{1}{M}F^* F-\frac{L_1}{N} I_N\|_\mathcal{F}^2) =  \frac{1}{M}\big(L_4-\frac{L_2}{N}\big),
\end{equation}
where $L_1:=\frac{1}{M}\sum_{k=1}^M M_2(\mu_k)$, $L_2:=\frac{1}{M}\sum_{k=1}^M M_2^2(\mu_k)$, and $L_4=\frac{1}{M}\sum_{k=1}^M M_4^4(\mu_k)$.
\end{theorem}

Under the notation of Theorem \ref{theorem:final main result}, the special case of probabilistic unit norm tight frames  was also addressed in \cite{me11}:
\begin{corollary}\label{corollary:approx}
Let $\{X_{k}\}_{k=1}^M$ be a collection of random vectors, independently distributed according to probabilistic unit norm tight frames $\{\mu_{k}\}_{k=1}^M$, respectively, such that $M_4(\mu_k)<\infty$. If $F$ denotes the random matrix associated to the analysis operator of $\{X_{k}\}_{k=1}^M$, then 
\begin{equation}\label{eq:in last theorem}
E(\|\frac{1}{M}F^* F-\frac{1}{N} I_N\|_\mathcal{F}^2) =  \frac{1}{M}\big( L_4-\frac{1}{N}\big),
\end{equation}
where $L_4=\frac{1}{M}\sum_{k=1}^M M_4^4$.
\end{corollary}

Randomness is used in compressed sensing to design suitable measurements matrices. Each row of such random matrices is a random vector whose construction is commonly based on Bernoulli, Gaussian, and sub-Gaussian distributions. We shall explain that these random vectors are induced by probabilistic tight frames, and in fact, we can apply Theorem \ref{corollary:approx}: 
\begin{example}\label{example:labeling}
Let $\{X_{k}\}_{k=1}^M$ be a collection of $N$-dimensional random vectors such that each vector's entries are independently identically distributed (i.i.d)~according to a probability measure with zero mean and finite $4$-th moments. This implies that each $X_{k}$ is distributed with respect to a probabilistic tight frame whose $4$-th moments exist. Thus, the assumptions in Theorem \ref{theorem:final main result} are satisfied, and we can compute \eqref{eq:theorem final} for some specific distributions that are related to compressed sensing:
\begin{itemize}
\item If the entries of $X_{k}$, $k=1,\ldots,M$, are i.i.d.~according to a Bernoulli distribution that takes the values $\pm \frac{1}{\sqrt{N}}$ with probability $\frac{1}{2}$, then $X_{k}$ is distributed according to a normalized counting measure supported on the vertices of the $d$-dimensional hypercube. Thus, $X_{k}$ is distributed according to a probabilistic unit norm tight frame for $\R^N$.
\item If the entries of $X_{k}$, $k=1,\ldots,M$, are i.i.d.~according to a Gaussian distribution with $0$ mean and variance $\frac{1}{\sqrt{N}}$, then $X_{k}$ is distributed according to a multivariate Gaussian probability measure $\mu$ whose covariance matrix is $\frac{1}{N} I_N$, and $\mu$ forms a probabilistic tight frame for $\R^N$. Since the moments of a multivariate Gaussian random vector are well-known, we can explicitly compute $L_4=1+\frac{2}{N}$, $L_1=1$, and $L_2=1$ in Theorem \ref{theorem:final main result}. Thus, the right-hand side of \eqref{eq:theorem final} equals $\frac{1}{M}(1+\frac{1}{N})$. 
\end{itemize}
\end{example}

\begin{acknowledgement}
Martin Ehler was supported by the NIH/DFG Research Career Transition Awards Program (EH 405/1-1/575910).  
K.~A.~Okoudjou was supported by ONR grants: N000140910324 $\&$ N000140910144,  by  a RASA from the Graduate School of UMCP and by the Alexander von Humboldt foundation. He would also like to express its gratitude to the Institute for Mathematics at the University of Osnabrueck. 

\end{acknowledgement}


\begin{thebibliography}{99.}%
%
%
%
%
%
\bibitem{albini09}
Albini, P., De Vito, E., Toigo, A.: Quantum homodyne tomography as an informationally complete positive-operator-valued measure. J. Phys. A \textbf{42} (29), pp.12,  (2009). 



\bibitem{ags05}
Ambrosio, L., Gigli, N., Savar\'e, G.: Gradients Flows in Metric Spaces and in the Space of Probability Measures. Lectures in Mathematics ETH Z\"urich. Birkh\"auser Verlag, Basel (2005).

\bibitem{Ball:1992fk}
Ball, K.: Ellipsoids of maximal volume in convex bodies. Geom.~Dedicata \textbf{41} no.~2, 241--250 (1992). 

\bibitem{bf03}
Benedetto, J.J.,  Fickus, M.: Finite normalized tight frames.  Adv. Comput. Math. \textbf{18}(2--4), 357--385 (2003).

\bibitem{Bodmann:2010fk}
Bodmann, B.~G., Casazza, P.~G.: The road to equal-norm Parseval frames.  J.~ Funct.~ Anal. \textbf{258}(2--4), 397-420 (2010).


\bibitem{Bourgain:1986aa}
Bourgain, J.: On high-dimensional maximal functions associated to convex bodies. Amer. J. Math., \textbf{108}(6), 1467--1476 (1986).

\bibitem{Cahill:2011}
Cahill, J., Casazza, P.~G.: The Paulsen problem in operator theory. 
submitted to Operators and Matrices (2011).

\bibitem{Casazzaa:2010fk}
Casazza, P.~G., Fickus, M., Mixon, D.~G.: Auto-tuning unit norm frames. Appl.~ Comput.~ Harmon.~ Anal., doi:10.1016/j.acha.2011.02.005, (2011).

\bibitem{Christensen:2003aa}
Christensen, O.: {A}n {I}ntroduction to {F}rames and {R}iesz {B}ases. Birkh{\"{a}}user, Boston (2003). 



\bibitem{Christensen:2003ab}
Christensen, O., Stoeva, D.T.: $p$-frames in separable {B}anach spaces. Adv. Comput.  Math., \textbf{18}, 117--126 (2003). 

\bibitem{ebdavies}
Davies, E.B.: Quantum theory of open systems. Academic Press, London--New York, (1976).

\bibitem{dl70}
Davies, E.B., Lewis, J.T.: An Operational Approach to Quantum Probability. Commun. Math.
Phys. \textbf{17}, 239--260 (1970).


\bibitem{Del77}
Delsarte, P., Goethals, J.~M., Seidel, J.~J.: Spherical codes and designs. Geom.~ Dedicata, \textbf{6}, 363--388 (1977).


\bibitem{me11} Ehler, M.:  Random tight frames.  J. Fourier Anal. Appl., doi:10.1007/s00041-011-9182-5 (2011). 

\bibitem{mejg11} Ehler, M., Galanis, J.:  Stat.~ Probabil.~ Lett. \textbf{81}, 8, 1046--1051 (2011). 

\bibitem{meko11} Ehler, M., Okoudjou, K.A.: Minimization of the probabilistic $p-$frame potential. Submitted for publication. 

\bibitem{gimi00}
Giannopoulos, A.A., Milman, V.D.: Extremal problems and isotropic positions of convex bodies. 
Israel J. Math. \textbf{117}, 29--60 (2000). 

\bibitem{John:1948uq}
John, F.: Extremum problems with inequalities as subsidiary conditions.
Courant Anniversary Volume, Interscience, 187--204 (1948). 

\bibitem{Kent:1988kx}
Kent, J.~T., Tyler, D.~E.: Maximum likelihood estimation for the wrapped {C}auchy distribution. Journal of Applied Statistics, \textbf{15} no.~2, 247--254 (1988).



\bibitem{koche1}
Kova\v{c}evi\'{c}, J., Chebira, A.: Life Beyond Bases: The Advent of Frames (Part I).
Signal Processing Magazine, IEEE, Volume \textbf{ 24},  Issue 4,  July 2007, 86--104.  

\bibitem{koche2}
Kova\v{c}evi\'{c}, J., Chebira, A.: Life Beyond Bases: The Advent of Frames (Part II).
 Signal Processing Magazine, IEEE, Volume \textbf{24}, Issue 5,  Sept. 2007, 115--125. 


\bibitem{B.Klartag:2009aa}
Klartag, B.: On the hyperplan conjecture for random convex sets. Israel J. Math., \textbf{170}, 253--268 (2009).

\bibitem{Levina:2001fk}
Levina, E., Bickel, P.: The Earth Mover's distance is the Mallows distance: some insights from statistics, Eighth IEEE International Conference on Computer Vision, \textbf{2}, 251--256 (2001). 

\bibitem{lyz07}
Lutwak, E., Yang, D., Zhang, G.: Volume inequalities for isotropic measures. Amer.~J.~Math., \textbf{129} no.~6, 1711--1723 (2007).  


\bibitem{Mardia:2008aa}
Mardia, K.~V., Peter, E.~J.: Directional Statistics, John Wiley \& Sons, Wiley Series in Probability and Statistics (2008).

\bibitem{Milman:1987aa}
Milman, V., Pajor, A.: Isotropic position and inertia ellipsoids and zonoids of the unit ball of normed $n-$dimensional space. In Geometric aspects of functional analysis, Lecture Notes in Math., pp. 64--104, Springer, Berlin, 1987-88.

\bibitem{krp67}
Parthasarathy, K. R.: Probability Measures On Metric Spaces. Probability and Mathematical Statistics, No. 3 Academic Press, Inc., New York-London (1967). 

\bibitem{Rene04}
Renes, J.~M, et al.: Symmetric informationally complete quantum measurements. Journal of Mathematical Physics, \textbf{45}, 2171--2180 (2004).

\bibitem{Tyler:1987fk}
Tyler, D.~E.: A distribution-free {$M$}-estimate of multivariate scatter. Annals of Statistics, \textbf{15} no.~1, 234--251 (1987).

\bibitem{Tyler:1987uq}
Tyler, D.~E.: Statistical analysis for the angular central {G}aussian distribution. Biometrika, \textbf{74} no.~3, 579--590 (1987).


\bibitem{vill09}
Villani, C.: Optimal transport: Old and new. Grundlehren der Mathematischen Wissenschaften, \textbf{338}, Springer-Verlag, Berlin (2009).

\bibitem{Wal03}
Waldron, S.: Generalised Welch bound equality sequences are tight frames. IEEE
Trans. Inform. Theory,  \textbf{49}, 2307--2309 (2003).
%


\end{thebibliography}
\end{document}